\documentclass[11pt]{amsart}
\usepackage{amsmath,amssymb,amsbsy,amsfonts,amsthm,latexsym,amsopn,amstext,amsxtra,epic, euscript,amscd,indentfirst}
\begin{document}

\newtheorem{theorem}{Theorem}
\newtheorem{conjecture}[theorem]{Conjecture}
\newtheorem{proposition}[theorem]{Proposition}
\newtheorem{question}[theorem]{Question}
\newtheorem{lemma}[theorem]{Lemma}
\newtheorem{cor}[theorem]{Corollary}
\newtheorem{obs}[theorem]{Observation}
\newtheorem{proc}[theorem]{Procedure}
\newcommand{\comments}[1]{} 
%% DEFINITIONS
\def\Z{\mathbb Z}
\def\Za{\mathbb Z^\ast}
\def\Fq{{\mathbb F}_q}
\def\R{\mathbb R}
\def\N{\mathbb N}
\def\k{\kappa}

\title{A Condition Forcing Spatial Curves to Develop Type-II Singularities Under Curve Shortening Flow}

\author{Gabriel J. H. Khan}
\email{khan.375@osu.edu}

\date{\today}

\maketitle 

%In this we seek to show how the Hamilton-Gage-Grayson theorem fails in three-dimensions.

%We do this in the following steps.
%\\
%1. Curves without flat points past the final inflection points develop type 2 singularities.\\
%2. If a single inflection point arises then it cannot mess this up but multiple inflection points arising simultaneously can.\\
%3. Propagation lemma - "derive the curvature kernel" This will show that the pictures being drawn actually compose proofs of self intersections happening.\\

\begin{abstract}
We show that if a curve that any curve immersed in $\R^3$ has everywhere positive torsion after the final inflection point, it develops a type-II singularity under curve shortening flow.
\end{abstract}

\section{Introduction}

Let $\gamma$ be a smooth immersion from $S^1$ to $R^n$. We then define the following differential equation.
\begin{equation}\label{eq:CSF}
\partial_t \gamma = \kappa N
\end{equation}
where $\kappa$ is the curvature and $N$ is the unit normal vector. We will study solutions to this equation, which consist of a family of curves $\gamma_t$ with $t \in [0,~\omega)$ which satisfy ~\eqref{eq:CSF} with $\gamma_0 = \gamma$.

This is commonly referred to as curve shortening flow and is the simplest example of mean curvature flow.
Michael Gage and Richard Hamilton proved short time existence and analyticity of the flow \cite{GH} and Matthew Grayson proved that the flow continues so long as curvature remains bounded \cite{G}. However, since the flow is the $L^2$ gradient flow for length of the curve (hence the name), a singularity must emerge as some time $\omega$.
A blow-up singularity is Type I if $\lim_{t \to \omega} M_{t} \cdot (\omega - t)$ is bounded and Type II otherwise where \begin{equation} M_{t} = \sup_{p \in \gamma_t} \kappa^2(p). \end{equation} We use the results of \cite{A} throughout and assume familiarity with this work. 

This is intended to show how strongly Grayson's theorem fails in higher dimensions. In two dimension,  Grayson proved that any initially embedded curve becomes convex. Combined with the earlier results of Gage and Hamilton, this shows that any initially embedded curve remains embedded, becomes convex and shrinks to a round point. In three dimensions, this is very different. Initially embedded curves can intersect and this theorem comes very close to showing that some embedded curves develop cusp-like singularities

\section{Positive Torsion Forces Type II Singularities}

\begin{theorem}
\label{Type II singularities}
Given a curve $\gamma$ in $\R^3$ without flat points past the last inflection point, under the curve shortening flow $\gamma_t$ develops a type II singularity.
\end{theorem}

\begin{proof}
Suppose $\gamma_0$ does not lie in any plane. Then, for all $t$ where the flow is defined, the torsion $\tau$ is non-zero as the flow is real-analytic so without loss of generality, the torsion is not everywhere zero.

Now suppose $\gamma$ develops a type I singularity at time $\omega$. Then, under renormalization as described in \cite{A}, $\gamma_t$ approaches an Abresch-Langer solution \cite{AL} with finite winding number in the $C^\infty$ sense \cite{A}. For any such Abresch-Langer curve $S$, $\sup_{p \in S} \k(p) \cdot L < \infty$ where $L$ is the length of S, i.e. $$L= \int_S ds.$$ 
Since $\gamma$ converges to some S, the functional $D(t) = \sup \k_t \cdot L_t$ for $t \in [0, \omega)$ and converges to a finite limit at $\omega$. Let $D= \lim_{t \to \omega} D(t)$.

Furthermore, given that $\gamma$ develops a type-I singularity, all blow up sequences are essential and since the curve converges to an Abresch Langer solution, any sequence ${p_m, t_m}$ such that $\lim_{m \to \infty} t_m = \omega$  , by the rough planarity theorem of \cite{A}, $$\lim_{t \to \omega} \sup_{p \in \gamma_t} \frac{\tau}{\k}(p) = 0$$ where $\tau$ is the torsion of $\gamma_t(p)$. 

A consequence of the fact that all sequences are essential is that there exists a time $c \in [0, \omega)$ such that for all $t \in [c, \omega)$, $\gamma_t$ has no inflection points so torsion is defined everywhere on the curve. Suppose torsion is positive for all time after $c$.

Now we consider the $L^1$ norm of torsion on the curve $\gamma_t$ for $t \in [c, \omega)$.

$$||\tau||_1 = \int_{\gamma_t} |\tau | \, ds$$

We parametrize $\gamma_t$ smoothly by $u \in [0, 2\pi)$ and utilize the calculations from \cite{A} and \cite{AG} for the following few calculations.

$$\int_{\gamma_t} |\tau | \, ds = \int_0^{2\pi}|\tau| \cdot v \, du \textrm{ where } v^2  = \langle \partial_u \gamma_t,\partial_u \gamma_t\rangle.$$

Taking the derivative with respect to time, we obtain
$$\partial_t\int_{\gamma_t} |\tau | \, ds =  \partial_t \int_0^{2\pi}(|\tau|\cdot v)\, du = \int_0^{2\pi} \partial_t (|\tau|\cdot v)\, du $$
We then calculate this explicitly.

\begin{eqnarray*}
\int_\gamma \partial_t |\tau|\cdot v\, du & = & \int_\gamma \partial_t (\tau\cdot v) \, du \\
& = & \int_0^{2\pi} (\partial_t \tau) \cdot v + (\partial_t v) \cdot \tau \, du\\
& = &  \int_0^{2\pi} \left( 2\k^2\tau + \partial_s \left(\frac{2\tau}{\kappa} 
\partial_s \kappa \right) + \partial_s^2 \tau) \right) \cdot v \\ 
& & + -\k^2 \, v \, \cdot\tau \, du\\
& = &  \int_0^{2\pi} \k^2\, \tau \, v \, du + \int_0^{2\pi}  \left(\partial_s \left(\frac{2\tau}{\kappa} 
\partial_s \kappa \right) + \partial_s^2 \tau\right) \cdot v \, du \\
 & = & \int_\gamma \k^2\, \tau \, ds + \int_\gamma \partial_s \left(\frac{2\tau}{\kappa} 
\partial_s \kappa \right) + \partial_s^2 \tau \, ds\\
& = & \int_\gamma \k^2\, \tau \, ds + \left(\frac{2\tau}{\kappa} 
\partial_s \kappa  + \partial_s \tau\right)|_{\partial \gamma}\\ \end{eqnarray*}

Since the boundary of $\gamma$ is empty,
$$ \partial_t ||\tau||_1 = \int_{\gamma_t} \k^2\cdot |\tau| \, ds > 0.$$

Therefore the $L^1$ norm is positive and increasing and so approaches a positive (possibly infinite) limit as t goes to $\omega$. However, $$\sup_{p \in \gamma_t} \tau (p)\cdot L_t \geq ||\tau||_1(t) > 0 $$ so 
$\liminf_{t \to \omega}\sup_{p \in \gamma_t} \tau (p)\cdot L_t \geq \lim_{t \to \omega} ||\tau||_1(t) = C > 0.$

But then $$\lim_{t \to \omega}\sup_{p \in \gamma_t} \frac{\tau}{\k} \geq \frac{C}{D} > 0. $$
$$\Rightarrow \Leftarrow$$

Therefore  $\gamma$ cannot develop a type I singularity and so develops a type II singularity.
\end{proof}

\section{On the Emergence of Flat Points}
A maximum principle argument shows that if a curve has no flat points, then it does not develop any under the curve shortening flow except possibly at inflection points.
Suppose $\tau(p,t)=0$ and this is the first time for which $\tau$ is ever non-positive. Then we compute the time derivative of $\tau$ at this point.
\begin{eqnarray*}
\partial_t \tau(p) & = & \partial_s^2 \tau + 2 \frac{1}{\k}(\partial_s \kappa) (\partial_s \tau) + \frac{2\tau}{\k} \left(\partial_s^2\k - \frac{1}{\k} (\partial_s \k)^2 + \k^3 \right) \\
& = & \partial_s^2 \tau + 2 \frac{1}{\k}(\partial_s \kappa) (\partial_s \tau) \textrm{ since $\tau = 0.$} \\
& = & \partial_s^2 \tau \geq 0 \textrm{ since $\partial_s \tau = 0$ and $p$ minimizes $\tau.$} \\ 
\end{eqnarray*}

This argument does not work at inflection points, where torsion is not defined. One could imagine that after an inflection point, a flat point could emerge where the inflection point was. In the future we will try to study this further and hopefully rule it out from occurring, which would give an open condition in the $C^3$ topology that ensures the curve develops a type 2 singularity.
%\\
%\\
%
%
% Since the inflection point disappears immediately and the torsion around the inflection point is positive by the previous argument, we translate the point slightly further in time and argue that it cannot remain flat. 
%\\
%\\
%\textit{ This is the key step and I need to be careful here}
%\\

 %Let $i_j$ be the $j$-th time for which $\gamma_t$ has an inflection point. Without loss of generality, we will consider $i_1$.
 %Let $i_2 - i_1 > \epsilon > 0$ and suppose $\tau(p, i_1 + \epsilon) = 0$. This is enough to consider, since none of the other points can be flat points as they have positive torsion.

 %If $\tau(p, i_1 + \epsilon) = 0$, then $\tau(p, i_1 + \epsilon/2) = 0$ as well by the previous argument. We now observe that $\partial_t^n \tau(p, i_1 + \epsilon/2) \neq 0$ for all $n$ since $\partial_s^n \tau(p, i_1 + \epsilon/2) \neq 0$ for all n since the curve is real analytic and $(p, i_1 + \epsilon/2)$ is a strict minimum of $\tau$ as the points around it have positive torsion. Therefore, $\tau(p, i_1 + \epsilon) \neq 0.$ $\Rightarrow \Leftarrow$

%As a result, if a curve develops a type-I singularity, there are at least two points for which torsion vanishes and the torsion has positive and negative sections. Furthermore, for any type I singularity, the $L^1$ norm of torsion must go to zero at the singularity. 

\section{Going Forward}
Various people have brought our attention to the following theorem. A proof is given by He Siming in \cite{H}.
\begin{theorem}
Given a curve $\gamma$ embedded on a standard sphere $S^2$ in $\R^3$, $\gamma_t$ remains spherical under ~\eqref{eq:CSF} and approaches a round point in the limit.
\end{theorem}

Combining these results is an extremely inefficient way to show that any curve embedded on a sphere contains a flat point (a much simpler observation is that the total torsion is zero). In fact, Sedykh's theorem states that in this case, there are at least four such points.


\begin{thebibliography}{99}
\bibitem{AL} U.~Abresch and J.~Langer, The Normalized Curve Shortening Flow and Homothetic Solutions, \emph{J. Differential Geom},\textbf{23} (1986),  175--196

\bibitem{A} S.~J.~Altschuler, Singularities of the Curve Shrinking Flow for Space Curves, \emph{J. Differential Geom} 
\textbf{34} (1991), 491--514.

\bibitem{AG} S.~J.~Altschuler and M.~A.~Grayson, Shortening Space Curves and Flow Through Singularities, \emph{J. Differential Geom} 
\textbf{35} (1992), 283--298.

\bibitem{GH} M.~Gage and R.~S.~Hamilton, The Heat Equation Shrinking Convex Plane Curves,\emph{J. Differential Geom} \textbf{23} (1986), 69--96.

\bibitem{G} M.~A.~Grayson, The Heat Equation Shrinks Embedded Plane Curves To Round Points,\emph{J. Differential Geom} \textbf{26} (1987), 285--314.

\bibitem{H} H.~Siming, Distance Comparison Principle and Grayson Type Theorem in the Three Dimensional Curve Shortening Flow

\bibitem{S} V.~D.~Sedykh, Four vertices of a convex space curve, Bull. London Math. Soc. 26 (1994), 177-180.

\end{thebibliography}
\end{document}